\documentclass[a4paper,12pt]{article}
\usepackage[utf8]{inputenc}
\usepackage[T1]{fontenc}
\usepackage[english]{babel}
\usepackage{amsmath, amsfonts, amssymb, mathtools}
\usepackage{graphicx}
\usepackage{mathrsfs}
\usepackage{dirtytalk}
\usepackage{xcolor}
\usepackage{chngcntr}
\usepackage{dsfont}
\usepackage{mathrsfs}
\usepackage{amsthm}

\usepackage[total={5.2in,9.1in},
top=1.4in, left=1.55in, includefoot]{geometry}

\newcommand{\R}{\mathbb{R}}

\newcommand{\Z}{\mathbb{Z}}
\newcommand{\N}{\mathbb{N}}

\newcommand{\supp}{\text{supp}}

\newtheorem{theorem}{Theorem}

\newtheorem{proposition}{Proposition}

\newtheorem{lemma}{Lemma}

\newtheorem{corollary}{Corollary}

\newtheorem{definition}{Definition}

\newtheorem{remark}{Remark}

\makeatletter
\def\moverlay{\mathpalette\mov@rlay}
\def\mov@rlay#1#2{\leavevmode\vtop{%
   \baselineskip\z@skip \lineskiplimit-\maxdimen
   \ialign{\hfil$\m@th#1##$\hfil\cr#2\crcr}}}
\newcommand{\charfusion}[3][\mathord]{
    #1{\ifx#1\mathop\vphantom{#2}\fi
        \mathpalette\mov@rlay{#2\cr#3}
      }
    \ifx#1\mathop\expandafter\displaylimits\fi}
\makeatother

\counterwithin{theorem}{section}

\counterwithin{proposition}{section}

\counterwithin{lemma}{section}

\counterwithin{corollary}{section}

\counterwithin{definition}{section}

\counterwithin{remark}{section}

\title{Singular Suspension Flows and Infinite Topological Entropy}
\author{Jonatas Marinho S. Araujo and Sergio Roma\~na}
\date{}
\begin{document}

\maketitle

We construct a dense set of homeomorphisms with infinite topological entropy whose associated pseudo-singular suspension flows have finite entropy—and indeed, arbitrarily small positive values can be achieved—showing that infinite entropy is not preserved under singular time changes. Complementing this, we prove that for a residual set of homeomorphisms, all pseudo-singular suspensions retain positive entropy. We also prove that for any $n\geq 2$, there exists a compact n-dimensional manifold admitting a minimal homeomorphism with infinite topological entropy; for such minimal homeomorphisms, a suitably chosen pseudo-singular suspension with a single singularity has zero entropy. Our results reveal that while positivity of entropy is generically stable, its infinitude is fragile under singular reparametrizations.


\section{Introduction}
In its most general formulation, Dynamical Systems is the study of group actions on spaces. Its origins can be traced to Henri Poincaré's investigations in celestial mechanics, where both continuous-time flows and discrete-time maps—and the interplay between them—already appeared. This interplay naturally raises the following fundamental question: which dynamical properties are preserved when passing from flows to maps, and vice versa?

The principal tools relating these two types of actions are the Poincaré map derived from a flow and the suspension flow constructed from a map. One then asks: what properties are preserved under these operations? In this work, we examine this question specifically for suspension flows and their entropy.

Entropy is a central concept in dynamical systems, measuring the exponential complexity of orbits. Positive entropy provides a quantitative notion of chaos. Two main notions exist: \emph{topological entropy} for continuous systems and \emph{metric entropy} for measure-preserving systems. The Variational Principle connects them, stating that topological entropy is the supremum of metric entropies over all invariant probability measures. For classical suspension flows, Abramov's formula shows that entropy is preserved: the suspension flow has positive (topological or metric) entropy exactly when the underlying map does. Thus, suspension flows are as complex as the maps from which they are derived.

Over the past decades, much work has been devoted to suspension flows and, more generally, to equivalent flows—flows sharing the same orbits. Unlike the map case, flows present an additional feature: singularities, i.e., fixed points of the flow. Flows without singularities are called nonsingular. Totoki \cite{Totoki1966} proved that the sign of metric entropy is preserved for nonsingular equivalent flows: it is either zero, positive, or infinite. Ohno \cite{Ohno1980} established the analogous result for topological entropy. However, Ohno also showed in the same paper that one can find equivalent flows with the same orbits but different entropy magnitudes—one with zero entropy and the other with infinite entropy.

Ohno further asked whether such examples with different entropy signs could be found in the differentiable setting. Sun, Zhou, and others \cite{WTY2009} answered this affirmatively by considering minimal diffeomorphisms with positive topological entropy and constructing singular reparametrizations of their suspension flows, which they called \emph{singular suspension flows}. They produced two such flows with the same orbits, one with positive topological entropy and the other with zero entropy. In particular, this shows that singular suspension flows derived from the same map can have different entropy signs, contrasting with the classical suspension case.

Singular suspension flows are of independent interest. Rego and Roma\~na \cite{RomanaRego2025} studied their entropy for diffeomorphisms (smooth case) and showed that, generically, when the singularities form a compact and countable set, the singular flow has positive entropy if the diffeomorphism does, and zero entropy if the diffeomorphism does.

In the present work, we study pseudo-singular suspension flows of homeomorphisms. Although similar in nature to the smooth case, the reduction of regularity makes the continuous setting more delicate. \\
Yano \cite{Yano1980} proved that a residual set of homeomorphisms on compact manifolds of dimension at least two possesses infinite topological entropy. This raises the natural question:\\
\ \\
\noindent
{\textbf{Question:} If a homeomorphism has infinite entropy, is it true that the pseudo-singular suspension flow also has infinite entropy?}\\
\noindent

We answer the question above negatively. This is the content of the next theorem.

\begin{theorem}\label{thmmain}
    There exist a dense set of homeomorphisms $\mathfrak{R}_{0}$ of a closed manifold $M$ such that each $f \in \mathfrak{R}_{0}$ has infinite topological entropy and there exists pseudo-singular suspension flow obtained from $f$ with at most one singularity with finite topological entropy.
\end{theorem}

Moreover, the technique used in the proof shows that the topological entropy of the pseudo-singular suspension flow can be made arbitrarily small.

Complementing this, we prove an analogue of the theorem of Rego and Romaña \cite{RomanaRego2025} for homeomorphisms:

\begin{theorem}\label{thmcomp}
    There exists a residual set $\mathfrak{R}$ of $\text{Homeo}(M)$ such that for every $f \in \mathfrak{R}$ and $\alpha:M \to \R$ continuous such that $\alpha^{-1}(0)$ is compact and countable, the pseudo-singular suspension flow has positive entropy.
\end{theorem}

Our arguments are based on \cite{RomanaRego2025} to obtain the positivity of the entropy, using pseudo-horseshoe structures and invariant measures to derive lower bounds.  We note that $\mathfrak{R}_{0} \subset \mathfrak{R},$ so all pseudo-singular suspensions obtained in \ref{thmmain} have positive entropy. It is worth noting that the techniques developed in \cite{WTY2009} and in \cite{RomanaRego2025} only yield positivity of topological entropy, without providing any control over its upper bound. Moreover, the lower regularity in the continuous setting makes singularities significantly more delicate to treat than in the smooth case. \\

Finally, we turn our attention to minimal homeomorphisms. Minimal homeomorphisms with positive entropy are known to be rare in connected manifolds; a few examples exist in higher dimensions within the smooth category 
\cite{WTY2009}. Since these examples are differentiable, their topological entropy is necessarily finite. Consequently, the existence of smooth manifolds admitting minimal homeomorphisms with infinite topological entropy is a considerably more delicate matter. In our last result, we employ a surgical construction—combining the embedding theorem of Beguin-Crovisier-Le Roux \cite{FSF2007} with Grillenberger's construction of minimal Cantor systems with arbitrary entropy \cite{GR72}—to establish the following:

\begin{theorem}\label{existence of minimal homeo}
    For any $n\geq 2$, there exists a compact $n$-dimensional manifold admitting a minimal homeomorphism with infinite topological entropy.
\end{theorem}

Moreover, following the arguments of \cite{WTY2009}, we will show that for such minimal homeomorphisms, one can choose a pseudo-singular suspension with a single singularity whose topological entropy is zero.

\
\noindent

\textbf{Organization of the paper}: in section $2$, we provide the basic definitions necessary for the comprehension of the text, such as orbits of discrete and continuous dynamical systems, the formal definition of topological and metric entropy and homeomorphisms with infinite entropy. In section $3$, we recall the definition of the reparametrization of continuous flows. In this section, we discuss the definition of pseudo-singular reparametrizations of a continuous flow and some of their properties, such as the $\gamma$ function. In section $4$, we prove Theorem \ref{thmmain}, \ref{thmcomp} and \ref{existence of minimal homeo}.

\section{Preliminaries}
\noindent
We now recall some basic definitions.\\

Let $(M,d)$ be a metric space. We say $R \subset M$ is \textit{residual} or \textit{generic} when it is the countable intersection of open and dense sets. Unless stated otherwise, this is the meaning of this word throughout this text.

Unless stated otherwise, $M$ is a compact topological manifold and $f:M\to M$ is a homeomorphism. Let $d$ be a metric compatible with the topology of the manifold. Define

$$\tilde{d}(f,g) : = d_{C_{0}}(f,g) + d_{C_{0}}(f^{-1},g^{-1})$$
\noindent
for $f,g \in \text{Homeo}(M).$ This defines a metric on $\text{Homeo}(M).$ The topology generated by such metric is called the \textit{$C^{0}$ topology.}

\subsection{Discrete and continuous-time Dynamical Systems}

\

\noindent

Given $f \in Homeo(M)$ we define the \textit{orbit} of $x \in M$ under $f$ by

$$\mathcal{O}_{f}(x) = \{f^{n}(x): n \in \Z\},$$
\noindent
where $f^{n} = f \circ \cdots \circ f$ $n$ times if $n > 0,$ and $f^{n} = f^{-1} \circ \cdots \circ f^{-1}$ $-n$ times if $n < 0,$ and $f^{0} = Id_{M}.$

We say a family of homeomorphisms $\varphi = (\varphi^{t})_{t \in \R}$ of $M$ is called a \textit{flow} when $f^{t+s} = f^{t} \circ f^{s}$ for all $t,s \in \R.$ We define the orbit of $x \in M$ under $\varphi$ by

$$\mathcal{O}_{\varphi}(x) = \{\varphi^{t}(x): t \in \R\},$$

We say $x \in M$ is a \textit{singularity} when $\varphi^{t}(x) = x$ for all $t \in \R.$ We denote the set of singularities of $\varphi$ by $Sing(\varphi).$ When $Sing(\varphi) = \varnothing$, we say the flow is \textit{nonsingular}, and when $Sing(\varphi) \neq \varnothing$, we say the flow is \textit{singular.}

\begin{definition}We say $(f,M)$ is minimal if the only compact invariant sets of $f$ are $\varnothing$ and $M.$ 

\end{definition}

Let $M$ be a compact metric space, let $f:M \to M$ be a homeomorphism, and let $r:M \to \R_{> 0}$ be a continuous function. We define

$$\overline{M} := \{(t,x) ; x \in M, t \in [0,r(x)]\}/\sim,$$
\noindent
where $(r(x),x)\sim (0,f(x)).$ It is known that $\overline{M}$ is a compact metric space. We also define $\pi:\overline{M}\to M$ to be the projection on the second coordinate.

\begin{definition}[Suspension Flows]
    We define the \textit{suspension flow of $f$ with roof function $r$} as the flow $(f^{t})_{t}:\overline{M} \to \overline{M}$ given by $f^{t}([s,x]) =(f^{n}(x),s'),$ where the integer $n$ e and real number $s'$ satisfy

$$\sum\limits_{j=1}^{n-1}r(f^{j}(x)) + s' = t+s, \hspace{0.5cm} 0 \leq s' \leq r(f^{n}(x)).$$

\end{definition}

The flow defined above is clearly continuous.\\

Now we define the topological entropy. Take $n \in \N.$ We define the metric $d_{n}$ in $M$ by $d_{n}(x,y) = \max\limits_{0 \leq i \leq n-1}d(f^{i}(x),f^{i}(y)).$ We say $X \subset M$ is $(n,\varepsilon)-$separated subset if $d_{n}(x,y) > \varepsilon$ for every $x,y \in X,$ where $x \neq y.$ We define $\text{sep}(n,\varepsilon,f)$ to be the maximal cardinality of a $(n,\varepsilon)-$separated subset.

\begin{definition}[Topological entropy of for maps]
We define the topological entropy of $f \in \text{Homeo}(M)$ by

$$h_{top}(f) = \lim\limits_{\varepsilon \to 0}\limsup\limits_{n\to \infty}\dfrac{1}{n}\text{sep}(n,\varepsilon,f).$$
\end{definition}

Now let $\varphi$ be a flow. 

\begin{definition}[Topological entropy for flows]
    We define the topological entropy of $\varphi$ as $h_{top}(\varphi) := h_{top}({\varphi^{1}}).$
\end{definition}

\

Now endow $M$ with the Borel $\sigma-$algebra. We say a probability measure in $M$ is $f-$\textit{invariant} when $\mu(f^{-1}(A)) = \mu(A)$ for all measurable sets $A \subset M.$ A $f-$invariant measure is \textit{ergodic} when $\mu(f^{-1}(A)\Delta A) = 0$ implies $\mu(A) \in \{0,1\},$ where $A \Delta B := (A \backslash B) \cup (B \backslash A)$ is the symmetric difference of the sets $A$ and $B.$ We denote by $\mathcal{M}(f)$ the set of $f-$invariant measures and $\mathcal{M}_{e}(f)$ the set of ergodic measures of $f.$

We also have a notion of entropy for systems with an invariant probability measure: Let $\mathcal{P}$ be a partition of $M$ into finite measurable sets and define $\mathcal{P}^{n} = \mathcal{P} \lor f^{-1}(\mathcal{P})\lor \cdots \lor f^{-n+1}(\mathcal{P})$, where $\lor$ means the refinement of the partitions. $\mathcal{P}^{n}$ is the partition of the itineraries of $f$ up to time $n$ subordinate to $\mathcal{P}.$ We define the metric entropy of a $f-$invariant measure $\mu$ by

$$h_{\mu}(f) = \sup\{h_{\mu}(f,\mathcal{P}): \mathcal{P} \text{ is a measurable and finite partition of }M\},$$

where $h_{\mu}(f,\mathcal{P}) := \lim\limits_{n \to \infty} \dfrac{1}{n}H(\mathcal{P}^{n})$ and $H(\mathcal{P}) = -\sum\limits_{P \in \mathcal{P}}\mu(P)\log(\mu(P)).$

The metric entropy and the topological entropy are related by the Variational Principle, which we will now state:

\begin{theorem}[Variational Principle]
    Let $f:M\to M$ be a continuous map. Then:
    $$h_{top}(f) = \sup\limits_{\mu \in M(f)}h_{\mu}(f) = \sup\limits_{\mu \in \mathcal{M}_{e}(f)}h_{\mu}(f).$$
\end{theorem}

\subsection{Homeomorphisms with infinite entropy}

\noindent

It is known that if $M$ is a closed manifold and $f:M\to M$ is smooth, then its topological entropy is finite. However, in the continuous case, we have a completely different typical phenomenon: infinite entropy. This is the content of Yano's Theorem which we now enunciate.

\begin{theorem}[Yano, 1980] \label{thmyano}

   Let $M$ be a compact manifold. If $\text{dim}(M) \geq 2,$ then for every positive number $K,$ the set $\tilde{E}_{K}(M) = \{f \in \text{Homeo}(M) ; h_{top}(f) \geq K\}$ contains an open and dense subset $\mathfrak{R}
   _{K}$ of $\text{Homeo}(M).$ Therefore, the set of homeomorphisms with infinite entropy contains $\mathfrak{R} = \bigcap_{k \geq 1}\mathfrak{R}_{k}$ which is residual in $\text{Homeo}(M).$ 

\end{theorem}

The Theorem above is accomplished by exhibiting a residual set consisting of homeomorphisms with infinitely many pseudo-horseshoes with an arbitrarily large number of legs. These homeomorphisms can be regarded as $C^{0}-$perturbations of homeomorphisms with possibly finite topological entropy. For instance, a typical homeomorphism in $\mathfrak{R}$ can be regarded as a $C^{0}-$perturbation of a homeomorphism in small balls such that, in each ball, we insert a pseudo-horseshoe, and outside these balls, the homeomorphism remains unchanged.

\section{Pseudo-singular Reparametrizations of Continuous Flows}

\noindent

We now describe the framework needed to treat singular flows arising from continuous flows. The main reference is \cite{Totoki1966}. Let $M$ be a compact metric space and let $(\phi^{t})_{t}:M\to M$ a continuous flow. Let $\alpha:M\to \R_{\geq 0}$ be a continuous function whose zero set $S := \alpha^{-1}(0) \neq \varnothing$ is compact and countable.

Define $$\theta(t,x) = \int_{0}^{t}\alpha(\phi^{s}(x))^{-1}ds,$$ for $t \in \R$, and $x \in M$ and define $$\tau(t,x) = \sup\{s ; \theta(s,x) \leq t\}.$$ It is known that $\theta$ is additive with respect to $\phi,$ just as $\tau$ (see \cite{Totoki1966}).

\begin{definition}(Pseudo-singular Flows)
    Define $\psi^{t}(x) := \phi^{\tau(t,x)}(x),$ $x \in M, t \in \R.$ The flow $\psi = (\psi^{t})_{t}$ is called the pseudo-singular suspension flow of $f$ with roof $r$ and brake $\alpha.$
\end{definition}

\begin{remark}

We use the word "pseudo-singular" because the lack of higher regularity may allow the reparametrized flow to be nonsingular. For example, if we consider the ODE $x' = x^{1\backslash3}$ with initial condition $x(0) = 0,$ it admits two solutions: $x(t) \equiv 0$ for all $t \in \R$ and $x(t) = (\frac{2}{3}t)^{2\backslash3}.$ The last solution is the solution given by the reparametrization above of the ODE $x' = 1$ and therefore by our definition the new flow has no singularities. On the other hand, $x' = x$ has a singularity at $0.$ This happens because $x \to x^{1/3}$ is merely a continuous function, while $x \to x$ is Lipschitz.
\end{remark}

\begin{remark}
    A sufficient condition to $\alpha^{-1}(0) = Sing(\psi)$ is the following: $\int_{0}^{t_{0}}\alpha(\phi^{s}(x))^{-1}ds = \infty,$ for every $x \in M$ and $t_{0} \in \R$ such that $\phi^{t_{0}}(x)\in S.$ 
\end{remark}

We have $Sing(\psi) \subset S.$ Since $\alpha$ is continuous, it has a maximum. Then $\frac{1}{\alpha}$ has a minimum greater than zero. Therefore, $\lim\limits_{t\rightarrow \infty} \theta(t,x) = \infty$ and $\lim\limits_{t\rightarrow -\infty} \theta(t,x) = -\infty,$ for all $x \in M.$ 

Given $x \in M\backslash\bigcup\limits_{t\in\R}\phi^{t}(S),$ we have that $\theta(\cdot,x)$ is an increasing homeomorphism from $\R$ to $\R$ fixing zero. Therefore, $\mathcal{O}_{\phi}(x) = \mathcal{O}_{\psi}(x).$ 

On the other hand, given $x \in \bigcup\limits_{t\in\R}\phi^{t}(Sing(\psi)),$ suppose $t_{0}$ is the real number closest to $0$ such that $\phi^{t_{0}}(x) \in S.$ It is clear that 

$$\lim\limits_{t\rightarrow \infty}\psi^{t}(x) = \phi^{t_{0}}(x).$$

In other words, the orbit of $x$ accumulates on the set of singularities. We now turn our attention to suspension flows. For instance, let $f:M \to M$ be a homeomorphism. Let $r:M \to \R_{>0}$ be a continuous function and let $(\phi^{t})_{t}:\overline{M}\to \overline{M}$ be the suspension flow of $f$ with roof function $r.$ Let $\alpha:\overline{M}\to \R_{\geq 0}$ be a continuous function.

\begin{definition}[Pseudo-singular suspension flows]
    Let $\psi$ be the pseudo-singular flow obtained from $\phi$ with brake $\alpha.$ We call $\psi$ the pseudo-singular suspension flow of $f$ with roof function $r$ and brake $\alpha.$
\end{definition}

Pseudo-singular suspension flows look like suspension flows because they are reparametrization of such flows and because they have a "roof function".

\begin{definition}[The function $\gamma$]
    Define $\gamma:M \to \R_{\geq 0}\cup\{+\infty\}$ as $\gamma(x) = \infty$ when $x \in \pi(Sing(\psi)),$ and otherwise $\psi^{\gamma(x)}(0,x) = \phi^{r(x)}(x) = (0,f(x)).$ 
\end{definition}

The function $\gamma$ serves as the new roof floor for the homeomorphism $f$ corresponding to the reparametrized flow $\psi,$ with a new set of points at infinity corresponding to the points of $\{0\}\times M$ that never return to $\{0\} \times M.$ This means that $\gamma$ is well defined only in $M \backslash \pi(Sing(S)).$ From the defining equation of $\gamma,$

$$\psi^{\gamma(x)}(x) = \phi^{\tau(\gamma(x),x)}(x) = \phi^{r(x)}(x).$$

Since both are the first return time for $x$ then we must have $r(x) = \tau(\gamma(x),x).$ Take $x \in M \backslash \pi(S).$ Observe that $t \rightarrow \theta(t,x)$ is an increasing homeomorphism and this implies that $\tau(\cdot,x)$ is the inverse function of $\theta(\cdot,x).$ Then $\gamma(x) = \theta(r(x),x).$ Therefore,

\begin{equation} \label{eq1}
    \gamma(x) = \int_{0}^{r(x)}\alpha(t,x)^{-1}dt,
\end{equation}

for all $x \in M \backslash \pi(Sing(\psi)).$

\begin{proposition}
    The function $\gamma:M \backslash \pi(Sing(\psi)) \to \R$ is continuous.
\end{proposition}

\begin{proof}
    Let $x,x_{0}$ be points in $M \backslash \pi(Sing(\psi)),$ $x_{0}$ fixed. We are going to show that $\gamma$ is continuous at $x_{0}.$ Observe that 

    $$\alpha(t,x)^{-1}-\alpha(t,x_{0})^{-1} = -\dfrac{\alpha(t,x)-\alpha(t,x_{0})}{\alpha(t,x)\alpha(t,x_{0})}.$$

    By the formula (\ref{eq1}), we have

\begin{equation*}
    \begin{split}
        |\gamma(x)-\gamma(x_{0})| & = \left| \int_{0}^{r(x)}\alpha(t,x)^{-1}dt - \int_{0}^{r(x_{0})}\alpha(t,x_{0})^{-1}dt\right|\\
        &\leq \left| \int_{0}^{r(x_{0})}(\alpha(t,x)^{-1} - \alpha(t,x_{0})^{-1}) dt\right| + \left|\int_{r(x_{0})}^{r(x)}\alpha(t,x)^{-1}dt. \right|
    \end{split}
\end{equation*}

Let $\varepsilon' > 0$ be an arbitrary number. The continuity of the flow implies that the map $\phi:[0,r(x_{0})]\times \tilde{M} \to \tilde{M}$ is continuous. This implies that $\alpha \circ \phi$ is continuous and hence uniformly continuous. From the uniform continuity of $\alpha$ take $\delta_{1} > 0$ such that $|\alpha(t,x) - \alpha(t,x_{0})| = |\alpha(\phi^{t}(x))-\alpha(\phi^{t}(x_{0}))| < \varepsilon'$ whenever $d(x,x_{0}) < \delta_{1}.$ Define $M_{0} = \sup\limits_{0 \leq t \leq r(x_{0})} \alpha(t,x)$ and $m_{0} = \inf\limits_{0\leq t \leq r(x_{0})}\alpha(t,x)$. For sufficiently small $\varepsilon':$

\begin{equation*}
    \begin{split}
        |\alpha(t,x)^{-1}-\alpha(t,x_{0})^{-1}| & = \left|\dfrac{\alpha(t,x)-\alpha(t,x_{0})}{\alpha(t,x)\alpha(t,x_{0})}\right|\\
        & = \left|\dfrac{\alpha(t,x)-\alpha(t,x_{0})}{(\alpha(t,x)-\alpha(t,x_{0}))\alpha(t,x_{0}) + \alpha(t,x_{0})^{2}}\right|\\
        & \leq \dfrac{\varepsilon'}{m_{0}^{2} - \varepsilon' M_{0}}.
    \end{split}
\end{equation*}

Therefore,

\begin{equation*}
    \begin{split}
        \left| \int_{0}^{r(x_{0})}(\alpha(t,x)^{-1} - \alpha(t,x_{0})^{-1}) dt\right| \leq \dfrac{\varepsilon \cdot r(x_{0})}{m_{0}^{2} - \varepsilon M_{0}}.
    \end{split}
\end{equation*}

Now we turn our attention to the second term.  We can take $\delta_{2}$ small enough so that $d(x,{x_{0}}) < \delta_{2}$ implies $|r(x)-r(x_{0})| < \varepsilon'.$
Also, fix $\varepsilon_{0} >0$ and $\delta_{0} >0.$ The function $(t,x) \in [r(x_{0})-\varepsilon_{0},r(x_{0})+\varepsilon_{0}] \times B_{\delta_{0}}(x) \to \alpha(\phi^{t}(0,x))^{-1}$ is (uniformly) continuous. Then, for sufficiently small $\varepsilon',$ there exists $\delta_{3} > 0$ such that $d(x,x_{0})< \delta_{3}$ implies

$$|\alpha(t,x)^{-1} - \alpha(t,x_{0})^{-1}| < \varepsilon'.$$

Put $M_{1} =  \sup\limits_{t \in [r(x_{0})-\varepsilon_{0},r(x_{0})+\varepsilon_{0}]}\alpha(t,x_{0})^{-1}.$ Then,

\begin{equation}
    \begin{split}
        \left|\int_{r(x_{0})}^{r(x)}\alpha(t,x)^{-1}dt\right| & \leq \left|\int_{r(x_{0})}^{r(x)}(\alpha(t,x)^{-1} - \alpha(t,x_{0})^{-1} )dt\right| + \left|\int_{r(x_{0})}^{r(x)}\alpha(t,x_{0})^{-1}dt\right|\\
        &  \leq \varepsilon' |r(x)-r(x_{0})| + M_{1}|r(x)-r(x_{0})| \\
        & \leq (\varepsilon')^2 + M_{1}\varepsilon'.
    \end{split}
\end{equation}

Therefore, 

$$|\gamma(x)-\gamma(x_{0})| \leq \dfrac{\varepsilon'}{m_{0}^{2} - \varepsilon' M_{0}} + (\varepsilon')^2 + M_{1}\varepsilon'.$$

The function $\varepsilon' \to \dfrac{\varepsilon'}{m_{0}^{2} - \varepsilon' M_{0}} + (\varepsilon')^2 + M_{1}\varepsilon'$ is continuous at $0.$ Then, given $\varepsilon> 0,$ choose $\varepsilon'$ small enough and $\delta = \min\{\delta_{1},\delta_{2},\delta_{3}\}$ accordingly such that $\dfrac{\varepsilon'}{m_{0}^{2} - \varepsilon' M_{0}} + (\varepsilon')^2 + M_{1}\varepsilon' < \varepsilon.$ Hence, $\gamma:M \backslash \pi(Sing(\psi))$ is continuous at $x_{0}.$ By the arbitrary choice of $x_{0},$ the proof is complete.
\end{proof}

Define

\begin{equation*}
A_{sing} = \pi\left(\bigcup\limits_{t \in \R}\phi^{t}(Sing(\psi))\right).   
\end{equation*}

$A_{sing}$ is countable. Indeed, by the definition of a suspension flow, $A_{sing}$ is the set of points of $M$ such that some vertical line $[0,r(x)] \times \{x\}$ contains a singularity. That is, it is the union of orbits of $f$ such that some vertical line based on an element of this orbit contains a singularity. Since the orbits are countable, so is $A_{sing}.$ 

In what follows, we present some useful lemmas that were stated in \cite{WTY2009} and \cite{RomanaRego2025} in the case of smooth flows. Here, they are stated and proved in general metric spaces. We write $E_{\mu}(f) := \int f d\mu$ to be the \textit{expected value of $f$ with respect to} $\mu.$
\begin{definition}

Let $\overline{\mu}$ be an ergodic measure of the flow $\psi.$ We say that $a \in \overline{M}$ is \textit{generic} for $\overline{\mu}$ under the action of $\psi$ when, for all continuous functions $f:\overline{M}\to \R$, we have 

$$\lim\limits_{t\to \infty}\dfrac{1}{2t}\int_{-t}^{t}f(\psi^{s}(a))ds = \int fd\overline{\mu}.$$

\end{definition}

Observe that if $a \in \overline{M}$ is \textit{generic} for $\overline{\mu}$, then $\psi^{t}(a)$ is generic for $\overline{\mu}$ for all $t \in \R.$ Moreover, by the Birkhoff Ergodic Theorem, we know that the set of generic points contains a set of full measure.

\begin{lemma}\label{lemmameas}

    For every ergodic and non-atomic measure $\overline{\mu}$ for $\psi$, there exists a unique ergodic and non-atomic measure $\mu$ such that

    $$E_{\overline{\mu}}(\varphi) = \dfrac{1}{E_{\mu}(\gamma)}E_{\mu}\left(\int_{0}^{\gamma(x)}\varphi(\psi^{t}(0,x))dt\right)$$
    for every continuous function $\varphi:\overline{M}\to\R.$
\end{lemma}

\begin{proof}
    We follow the ideas introduced in \cite{Ohno1980}, and later formulated in \cite{WTY2009}. We observe that in both papers, the essential point is that the argument applies to a continuous flow on a compact metric space.

    \begin{itemize}

    \item[(1)] \textit{Existence.} Let $(t,a)$ be a generic point for $\overline{\mu}.$ By the definition of a suspension flow, we may suppose that $t = 0.$ Since this point is generic, $a \notin A_{sing}$, which implies that $\lim\limits_{n\rightarrow \infty}\sum\limits_{k=1}^{n}\gamma(f^{k}(a)) =\infty.$ Again, by the definition of a suspension flow, changing $t$ to $\sum\limits_{k=1}^{n}\gamma(f^{k}(a)),$ we obtain

    $$E_{\overline{\mu}}(\varphi) = \lim\limits_{n\rightarrow \infty} \dfrac{1}{\sum\limits_{k=-n}^{n}\gamma(f^{k}(a))}\int_{\sum\limits_{k=-1}^{-n}\gamma({f^{k}(a))}}^{\sum\limits_{k=0}^{n}\gamma(f^{k}(a))}\varphi(\psi^{s}(0,a))ds=$$

    $$\lim\limits_{n\rightarrow \infty}\dfrac{1}{\sum\limits_{k=-n}^{n}\gamma(f^{k}(a))}\sum\limits_{k=-n}^{n}\int_{0}^{\gamma(f^{k}(a))}\varphi(\psi^{s}(0,f^{k}(a)))ds.$$

    Let $\mu$ be an accumulation point of $\left\{\dfrac{1}{2n}\sum\limits_{k=-n}^{n}\delta_{f^{k}(a)}; n> 0\right\}.$ By construction, $a$ is "generic" for $\mu$ with relation to $f$ and $\mu$ is non-atomic.

    Write $\pi(S) = \{a_{1},a_{2},\dots\}$ and choose $0 < \varepsilon_{n}^{m}<\eta_{n}^{m}$ such that $\lim\limits_{m} \eta_{n}^{m} =0,$ and choose bump functions $\varphi^{m}:M\to \R_{\geq 0}$ such that $0 \leq \varphi^{m} \leq 1,$ $\supp(\varphi^{m}) \subset M \backslash \left( \cup_{n} \overline{B_{\varepsilon_{n}^{m}}(a_{n})} \right)$ and $\varphi^{m} \equiv 1$ in $M \backslash \left( \cup_{n} \overline{B_{\eta_{n}^{m}}(a_{n})} \right),$ in a way that $(\varphi^{m})_{m}$ is increasing. In this setting, we observe that $\varphi^{m}\rightarrow 1$ pointwise in $M.$ We use this sequence of functions to deal with the points at infinity of $\gamma.$

    Define $\overline{\varphi}^{m}:\overline{M} \to \R$ by $\overline{\varphi}^{m}(t,x) = \varphi^{m}(x)$ for all $(t,x) \in \overline{M}.$ We observe that these continuous functions have the property that $\supp({\overline{\varphi^{m}}}) \subset \overline{M} \backslash S$ and $\overline{\varphi}^{m}\rightarrow 1$ pointwise in $\overline{M}.$ Moreover, define $\Phi^{m}:M \to\R$ by

    \begin{equation*}
        \Phi^{m}(x) = \int_{0}^{\gamma(x)}\overline{\varphi}^{m}(\psi^{s}(0,x))ds.
    \end{equation*} 
    
    By the definition of $\overline{\varphi}^{m}$ and by the property of its support, we have $\Phi^{m}(x)= \gamma(x)\varphi^{m}(x),$ which is continuous in $M$ (remember $\varphi^{m}$ is zero near the singularities). Since $\varphi^{m} \rightarrow 1$ pointwise, we must have 
    
    \begin{equation*}
        \int \varphi^{m}(x)\gamma(x)d\mu(x) > 0
    \end{equation*}
    
     for some $m.$ On the contrary, by taking the limit in $m$, we would conclude that $\int \gamma d\mu = 0.$ Hence, $\gamma \equiv 0 \mod \mu,$ and this cannot be true since $\mu$ is non-atomic and $\gamma$ has at most a countable number of zeros and $\gamma(f^{k}(a)) > 0$ for all $k \in \Z$, or if we simply suppose, without loss of generality, that $\gamma > 0,$ that is, if no singularity was inserted inside $M.$ Then, there exists $m_{0}$ such that $\int \Phi^{m_{0}}d\mu >0.$ We get

    $$\int \overline{\varphi}^{m_{0}}d\overline{\mu} = \lim\limits_{n'\rightarrow \infty}\dfrac{1}{\sum\limits_{k=-n'}^{n'}\gamma(f^{k}(a))}\sum\limits_{k=-n'}^{n'}\int_{0}^{\gamma(f^{k}(a))}\overline{\varphi}^{m_{0}}(\psi^{s}(0,f^{k}(a)))ds=$$

    $$=\lim\limits_{n'\rightarrow \infty}\dfrac{1}{\dfrac{1}{2n'}\sum\limits_{k=-n'}^{n'}\gamma(f^{k}(a))}\dfrac{1}{2n'}\sum\limits_{k=-n'}^{n'}\int_{0}^{\gamma(f^{k}(a))}\overline{\varphi}^{m_{0}}(\psi^{s}(0,f^{k}(a)))ds.$$

    Note that 
    
    $$\lim\limits_{n'\rightarrow \infty}\dfrac{1}{2n'}\sum\limits_{k=-n'}^{n'}\int_{0}^{\gamma(f^{k}(a))}\overline{\varphi}^{m_{0}}(\psi^{s}(0,f^{k}(a)))ds =$$
    
    $$\lim\limits_{n'\rightarrow \infty} \dfrac{1}{2n'}\sum\limits_{k=-n'}^{n'}\Phi^{m_{0}}(f^{k}(a)) = \int \Phi^{m_{0}}d\mu > 0,$$
    
    and this implies that the limit

    $$\lim\limits_{n'\rightarrow \infty}\dfrac{1}{2n'}\sum\limits_{k=-n'}^{n'}\gamma(f^{k}(a))$$

    exists. Call this limit $c.$ By the Monotone Convergence Theorem,

    $$1 = \lim\limits_{m\rightarrow \infty} E_{\overline{\mu}}(\overline{\varphi}^{m}) = \dfrac{1}{c}\lim\limits_{m \to \infty}E_{\mu}(\varphi^{m}\gamma) = \dfrac{1}{c}E_{\mu}(\gamma),$$

    that is, $E_{\mu}(\gamma) = c$ is finite.

   Let $\varphi$ be any continuous function in $\overline{M}.$ By the Dominated Convergence Theorem:

    \begin{equation*}
        \begin{split}
            E_{\overline{\mu}}(\varphi) =\lim\limits_{m \rightarrow \infty} E_{\overline{\mu}}(\varphi\cdot \overline{\varphi}^{m}) &=\lim\limits_{m \rightarrow \infty} \dfrac{1}{E_{\mu}(\gamma)}E_{\mu}\left(\int_{0}^{\gamma(x)}(\varphi\cdot \overline{\varphi}^{m})(\psi^{t}(0,x))dt\right)\\
            & = \dfrac{1}{E_{\mu}(\gamma)}E_{\mu}\left(\int_{0}^{\gamma(x)}\varphi(\psi^{t}(0,x))dt\right),
        \end{split}
    \end{equation*}

   and this shows the equality required.


    \item[(2)] \textit{$\mu$ is unique.} Indeed, given $A \subset M$ measurable, let $A_{0} = \{(t,x) \in \overline{M}; x \in A\}.$ Then, $\overline{\mu}(A_{0}) = \dfrac{1}{E_{\mu}(\gamma)}E_{\mu}\left(\int_{0}^{\gamma(x)}\mathds{1}_{A_{0}}(\psi^{s}(0,x))ds\right) = \dfrac{1}{E_{\mu}(\gamma)} E_{\mu}\left(\mathds{1}_{A}\gamma\right) = \dfrac{1}{E_{\mu}(\gamma)} \int_{A}\gamma d\mu,$ that is, given another measure $\mu'$ which satisfies the hypothesis of the lemma, we have that $\gamma d\mu = \gamma d\mu'.$ Since the measure $\mu$ does not possess atoms, excluding the points where $\gamma$ is infinity (which is a countable set and hence has zero measure) we have $d\mu = d\mu',$ that is, $\mu = \mu'.$ 
    
    \item[(3)] \textit{$\mu$ is ergodic.} To verify that $\mu$ is ergodic, let $A \subset M$ be $f-$invariant. Then, $A_{0}$ is measurable and $\psi-$invariant, and then $A_{0}$ has measure $0$ or $1$ with respect to $\overline{\mu}.$ Since $\overline{\mu}(A_{0}) = \dfrac{1}{E_{\mu}(\gamma)}\int_{A}\gamma d\mu$ and $\gamma^{-1}(0)$ has zero measure, $A$ must have measure $0$ or $1.$
    
    \end{itemize}
\end{proof}

\begin{corollary}
    If $E_{\mu}(\gamma) = \infty$ for all ergodic and non-atomic measures $\mu$ for $f$, then $\psi$ has only atomic measures.
\end{corollary}

\begin{proof}
    This corollary was obtained in \cite{WTY2009}. Observe that it also follows from the proof of the lemma above, producing another proof of this corollary. In fact, on the contrary, there would exist a non-atomic and ergodic measure $\overline{\mu}$ for $\psi.$ By the proof of the lemma, there exists a non-atomic and ergodic measure $\mu$ for $f$ such that $E_{\mu}(\gamma) < \infty$, and this contradicts the hypothesis.
\end{proof}

The following Theorem is a continuous version of Theorem $3.2$ in \cite{RomanaRego2025}.

\begin{theorem} 
    Let $f:M\to M$ be a homeomorphism in a compact metric space with positive entropy and $\psi$ a pseudo-singular suspension flow with roof $r$ and brake $\alpha.$ Then $h_{top}(\psi) = 0$ if and only if every ergodic measure $\mu$ of $f$ satisfying $h_{\mu}(f) > 0,$ we have $E_{\mu}(\gamma) = \infty.$
\end{theorem}
\noindent
\begin{proof}
\noindent
    We follow the steps in \cite{RomanaRego2025}, Theorem $3.2.$

    Suppose first that $E_{\mu}(\gamma) = \infty$ for every ergodic measure $\mu$ that satisfies $h_{\mu}(f) > 0.$ By the Variational Principle, to prove that $h_{top}(\psi) = 0$ it suffices to show that $h_{\overline{\mu}}(\psi) = 0$ for every ergodic and non-atomic measure of $\psi.$ Let $\overline{\mu}$ be such an ergodic and non-atomic measure of $\psi.$ By Lemma \ref{lemmameas}, there exists a unique ergodic and non-atomic measure $\mu$ of $f$ such that $E_{\mu}(\gamma) < \infty$ and
    
    \begin{equation}
        E_{\overline{\mu}}(\varphi) = \dfrac{1}{E_{\mu}(\gamma)}E_{\mu}\left(\int_{0}^{\gamma(x)}\varphi(\psi^{t}(0,x))dt\right).
    \end{equation}

    As observed earlier, $A_{sing}$ is countable, which implies that $\mu(A_{sing}) = 0$ since the measure is non-atomic. For $x \in M \backslash A_{sing},$ $t \rightarrow \theta(t,x)$ is increasing and differentiable, hence invertible. By putting $t = \theta(s,x),$ $dt = \dfrac{d}{ds}\theta(s,x)ds = \alpha(s,x)^{-1}ds,$ and then
    \begin{equation}
        \begin{split}
            E_{\overline{\mu}}(\varphi) &= \dfrac{1}{E_{\mu}(\gamma)}E_{\mu}\left(\int_{0}^{\gamma(x)}\varphi(\psi^{t}(0,x))dt\right)\\
            & =  \dfrac{1}{E_{\mu}(\gamma)}E_{\mu}\left(\int_{0}^{r(x)}\varphi(t,x)\alpha(t,x)^{-1}dt\right).
        \end{split}
    \end{equation}

    Define a new measure by $\hat{\mu}(B) := \dfrac{1}{E_{\mu}(r)}E_{\mu}(\int_{0}^{r(x)}\mathds{1}_{B}(t,x)dt).$ Observe that $E_{\hat{\mu}}(\alpha^{-1}) < \infty:$ indeed, $E_{\hat{\mu}}(\alpha^{-1}) = \dfrac{1}{E_{\mu}(r)}\int_{0}^{r(x)}\alpha(t,x)^{-1}dt = \dfrac{E_{\mu}(\gamma)}{E_{\mu}(r)} < \infty.$ Therefore, the measure $\hat{\mu}_{\alpha}$ defined by $\hat{\mu}_{\alpha}(B) = \int_{B}(1/\alpha) d\hat{\mu}$ is finite.
    
    \
    
    Moreover, $E_{\overline{\mu}}(\varphi) = \dfrac{E_{\mu}(r)}{E_{\mu}(\gamma)}E_{\hat{\mu}_{\alpha}}(\varphi),$ for all $\varphi$ continuous. Thus, a direct computation shows that $h_{\overline{\mu}}(\psi) = \dfrac{E_{\mu}(r)}{E_{\mu}(\gamma)} h_{\hat{\mu}_{\alpha}}(\psi).$ By proposition $2.15$ from \cite{WTY2009}, we have that $h_{\hat{\mu}}(\phi) = h_{\mu}(f) = 0,$  and by theorem $2.9$ from \cite{WTY2009} it follows that $h_{\hat{\mu}}(\psi) = 0.$ Then $h_{\overline{\mu}}(\psi) = 0.$ 

    \
    
     On the other hand, suppose that there exists an ergodic measure $\mu$ for $f$ such that $h_{\mu}(f) > 0$ and $\int \gamma d\mu < \infty.$ We define $\overline{\mu}$ by the equation

    \begin{equation}
        E_{\overline{\mu}}(\varphi) = \dfrac{1}{E_{\mu}(\gamma)}E_{\mu}\left(\int_{0}^{\gamma(x)}\varphi(\psi^{t}(0,x))dt\right),
    \end{equation}

     $\varphi:\overline{M}\to \R$ continuous. This clearly defines a measure, since it is a continuous functional defined in the space of continuous functions from $\overline{M}$ to $\R.$ Observe that it is a probability measure as well. At last, to show that $\overline{\mu}$ is $\psi-$invariant, we follow the steps of \cite{RomanaRego2025}. We use the second part of the argument above to conclude that $h_{\overline{\mu}}(\psi) > 0.$
    
\end{proof}






A good feature of pseudo-singular suspension flows is that the Abramov formula holds.

\begin{proposition}[Abramov Formula]\label{prop1}
    Let $\mu$ be a non-atomic and ergodic measure of $f$ and $\overline{\mu}$ the $\psi-$invariant obtained by lemma \ref{lemmameas}. Then $h_{\overline{\mu}}(\psi) = \dfrac{h_{\mu}(f)}{E_{\mu}(\gamma)}.$
\end{proposition}

\begin{proof}
    This follows from lemma $10.2$ of \cite{Totoki1966}. 
\end{proof}

\section{Proof of Main Results }

\noindent

Now we turn our attention to the study of the entropy of pseudo-singular suspension flows arising from homeomorphisms with infinite entropy.

Take $f \in \text{Homeo}(M)$ and let $r:M \to \R_{> 0}$ be a continuous function. Let $\phi = (\phi^{t})_{t}:\overline{M}\to \overline{M}$ be the suspension flow of $f$ with roof function $r.$ Let $\alpha:\overline{M} \to \R_{\geq 0}$ be a continuous function, and let $\psi = (\psi^{t})_{t}:\overline{M}\to \overline{M}$ be the pseudo-singular suspension flow with roof $r$ and brake $\alpha.$ From what we studied in earlier sections, $\psi$ is a time change of $\phi$ with possibly some singularities.
\subsection{Proof of Theorem  \ref{thmmain}}
Before proceeding to the actual proof, we will need some useful lemmas. They will be useful to control the growth of the topological entropy.

\begin{lemma}\label{lemmaentsup}
Let $X$ be a compact metric space and $f:X\to X$ be an endomorphism (or homeomorphism). Let $X = \bigcup\limits_{\lambda \in \Lambda}I_{\lambda}$ be the union of compact $f-$invariant subsets. Then $h_{top}(f) = \sup\limits_{\lambda \in \Lambda}h_{top}(f|_{I_{\lambda}}).$
\end{lemma}

Since the entropy of the flow is defined to be the entropy of the time$-1$ map, the same lemma holds true for flows.

\

The next lemma is a result originally discovered by Bowen (See \cite{B1970}).

\begin{lemma}\label{lemmawandering}

Let $f:X\to X$ be a continuous map on a compact metric space $X.$ Let $\Omega(f) \subset X$ be the non-wandenring set of $f.$ Then $h_{top}(f) = h_{top}(f|_{\Omega(f)}).$

\end{lemma}

We are ready to prove Theorem \ref{thmmain}

\begin{proof}[\emph{\textbf{Proof of Theorem \ref{thmmain}}}]

 Consider the set $\mathfrak{R}_{0}$ of $C^{0}-$perturbations of $C^{r}$ diffeomorphisms in $M$ as described in the commentary after Theorem \ref{thmyano}. For instance, we deal with the case of homeormorphisms that possess pseudo-horseshoes accumulating at one point and are $C^{r}$ outside the pseudo-horseshoes. Since $C^{r}$ diffeomorphisms are dense in $Homeo(M),$ the arguments used by Yano in \cite{Yano1980} show that $\mathfrak{R}_{0}$ is dense in $Homeo(M).$ We will follow a step-by-step proof in order to make the strategy of the proof clear.

\
\noindent
\underline{\textit{The homeomorphism.}}  Let $M$ be a closed compact manifold with dimension $n\geq 2.$ Let $\mathfrak{R}_{0}$ be the set of homeomorphisms of $M$ as in the remark after Theorem \ref{thmyano}. We will describe these homeomorphisms below in a more appropriate setting to our purposes.

Choose disjoint open sets $\{U_{i}\}_{i \in \N}$ in $M$ with $\text{diam}(U_{i})\overset{i \to \infty}{\longrightarrow} 0$ such that $(\overline{U}_{i})$ converge to a point $A$ with respect to the Hausdorff distance. Inside each $U_{i}$ choose another open set $V_{i} \subset U_{i}$ in a way that $\overline{V_{i}} \subsetneq U_{i}$ and which is homeomorphic to a disk. We construct $f$ with the following properties:

\begin{itemize}
    \item In each open set $U_{i}$, we have a horseshoe $\Lambda_{i} \subset V_{i}$ and $V_{i}$ is an isolating neighborhood for $\Lambda_{i}.$

    \item $f$ is $C^{r}$ in $V = M \backslash \displaystyle\bigcup_{i \in \N}U_{i}.$

    \item $A$ is a periodic point of $f.$  
\end{itemize}

\begin{figure}[ht]\label{fig1}
    \centering
    \includegraphics[scale = .47]{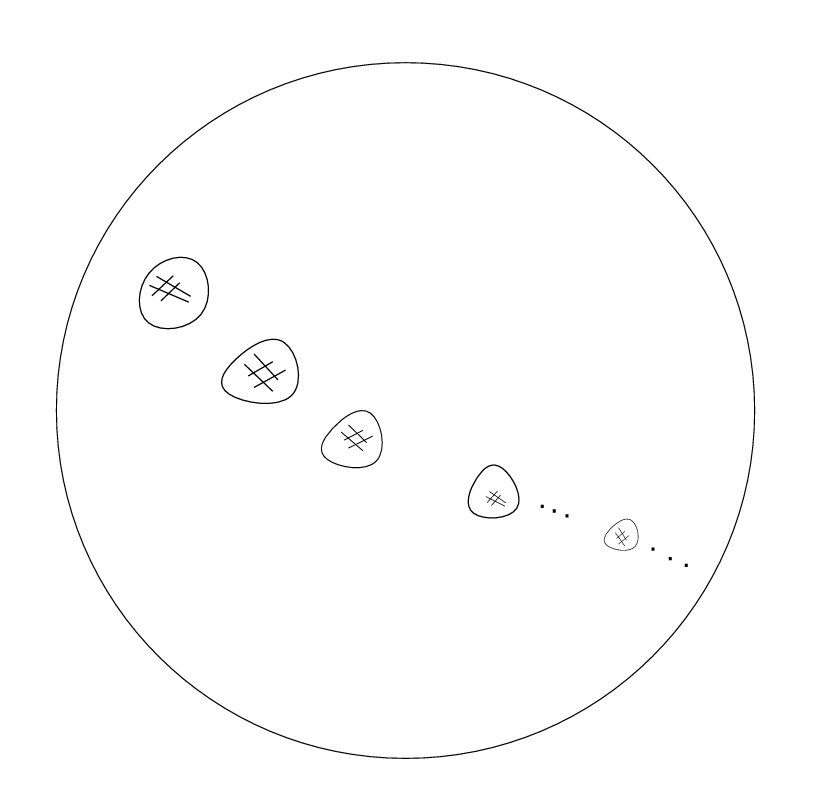}
    \caption{Homeomorphism with infinitely many horseshoes}
\end{figure}

Put $\Lambda_{V} = \displaystyle\bigcap_{i \in \Z}f^{i}(\overline{V}).$ Therefore, we conclude that $\Omega(f)$ is equal to $\Lambda_{V} \cup \left(\bigcup\limits_{i \in \N} \Lambda_{i}\right).$ The horseshoes are constructed in such a way that $(h_{top}(f|_{\Lambda_{i}}))$ is an increasing and unbounded sequence. Observe that $f$ is $C^{r}$ around $\Lambda_{V},$ and this implies that $h_{top}(f|_{\Lambda_{V}}) < +\infty.$

\

\underline{\textit{The pseudo-brake $\alpha.$}} Let $r:M \to \R_{>0}$ be a continuous function and define $(\phi^{t})_{t}:\tilde{M} \to \tilde{M}$ to be the suspension flow of $f$ with roof $r.$ Let $C = \displaystyle\inf_{x \in M} r(x) > 0.$ Let $\pi:\tilde{M}\to M$ be the projection on the second coordinate and put $\tilde{U}_{i} := \pi^{-1}(U_{i}),$ $\tilde{V}_{i} := \{[(t,x)] \in \tilde{M}: x \in V_{i} \text{ and } 0 \leq t \leq \frac{C}{i}\}.$ Now construct bump functions $\beta_i:\tilde{M}\to \R_{\geq 0}$ such that $\supp(\beta_{i}) \subset \tilde{U}_{i}$ and $\beta_{i}|_{\overline{\tilde{V}_{i}}} \equiv 1,$ for all $i \in \N.$ Choose numbers $0 \leq \gamma_{i} \leq 1,$ $i\in \N,$ and put $\alpha = 1 - \sum\limits_{i=1}^{\infty}(1-\gamma_{i})\beta_{i}.$ In this way, $\alpha \equiv 1$ on $\tilde{M} \backslash\bigcup\limits_{i \in \N}\tilde{U}_{i}$ and $\alpha|_{\tilde{V}_{i}} \equiv \gamma_{i}$ for every $i \in \N.$ We choose the $\gamma_{i}'s$ to be decreasing and $\displaystyle\lim_{i\to \infty}\gamma_{i}=0.$ In this way, $\alpha$ is continuous and $\alpha(0,A) = 0.$

Let $\psi$ be the pseudo-singular suspension flow of $f$ with pseudo-brake $\alpha.$ Put $\tilde{\Lambda}_{i}:= \pi^{-1}(\Lambda_{i})$ for all $i \in \N$ and $\tilde{\Lambda}_{V} := \pi^{-1}(\Lambda_{V}).$ By construction, we conclude that $\Omega(\psi) \subset \tilde{\Lambda}_{V} \cup \left(\bigcup\limits_{i \in \N} \tilde{\Lambda}_{i}\right).$

\underline{\textit{Control of the topological entropy of $\psi$ by $\gamma.$}} By equation (\ref{eq1}), for every $x \in V_{i}:$

\begin{equation*}
    \begin{split}
        \gamma(x) = \int_{0}^{r(x)}\alpha(t,x)^{-1}dt \geq  \int_{0}^{\frac{C}{i}}\alpha(t,x)^{-1}dt = \frac{C}{i}\gamma_{i}^{-1}.
    \end{split}
\end{equation*}

In particular, for every $f-$invariant measure $\mu_{i}$ supported on $\Lambda_{i},$ we conclude that $E_{\mu_{i}}(\gamma) \geq \dfrac{C}{i} \gamma_{i}^{-1}.$ Now take an $\psi-$invariant measure $\overline{\mu}_{i}$ which is ergodic a non-atomic and supported on $\tilde{\Lambda}_{i}$. By lemma \ref{lemmameas}, there exists an $f-$invariant, ergodic and non-atomic measure $\mu_{i}$ such that

$$E_{\overline{\mu}_{i}}(g) = \dfrac{1}{E_{\mu_{i}}(\gamma)}E_{\mu_{i}}\left(\int_{0}^{\gamma(x)}g(\psi^{t}(0,x))dt\right)$$

for every continuous function $g:\overline{M}\to\R.$ By Abramov formula, it follows from the construction of $\alpha$ that

\begin{equation*}
    \begin{split}
        h_{\overline{\mu}_{i}}(\psi) = \dfrac{h_{\mu_{i}}(f)}{E_{\mu_{i}}(\gamma)} \leq \dfrac{i}{C\gamma_{i}^{-1}} \cdot h_{top}(f|_{\Lambda_{i}}).
    \end{split}
\end{equation*}

Therefore we conclude that $h_{top}(\psi|_{\tilde{\Lambda}_{i}}) \leq \dfrac{h_{top}(f|_{\Lambda_{i}})\cdot i}{C\gamma_{i}^{-1}},$ for all $i \in \N.$ In the last two inequalities we used the Variational Principle.

By Lemma \ref{lemmawandering} we know that the entropy is concentrated in $\Omega(\psi).$ Therefore, by lemma \ref{lemmaentsup}, we have that

\begin{equation}
  h_{top}(\psi) =\max\left(\sup\limits_{i \in \N} h_{top}(\psi|_{\tilde{\Lambda}_{i}}), h_{top}(f|_{\tilde{\Lambda}_{V}})\right).
\end{equation}

Choose $(\gamma_{i})_{i}$ recursively by $\gamma_{1} = 1$ and $\gamma_{i+1} = \min\left\{\dfrac{\gamma_{i}}{2},\dfrac{C}{h_{top}(f|_{\Lambda_{i}})\cdot i}\right\}$ for all $i \in \N.$ Then $(\gamma_{i})_{i}$ is indeed a decreasing sequence and $\displaystyle\lim_{i \to \infty}\gamma_{i} = 0.$ Observe that $h_{top}(\psi|_{\tilde{\Lambda}_{i}}) \leq \dfrac{h_{top}(f|_{\Lambda_{i}})\cdot i}{C\gamma_{i}^{-1}}\leq 1$ for all $i \in \N.$ Then, let $K = \sup\{h_{top}(\psi|_{\tilde{\Lambda}_{V}}),1\}.$ This implies that $h_{top}(\psi) \leq K < +\infty.$

\end{proof}

\begin{remark}

If we change $\alpha$ by $\alpha_{\varepsilon} := \varepsilon \cdot \alpha$, we in fact show that $h_{top}(\psi) \leq \varepsilon \cdot K.$ Then, the topological entropy of $\psi,$ even though always positive by theorem \ref{thmcomp}, can be arbitrarily small.

\end{remark}

\begin{remark}
    We observe that the argument above to construct $f$ may fail to be residual since $f$ is a $C^{0}$ perturbation of a $C^{r}$ map. The key observation is that the set of $C^{r}$ maps is dense in the set of homeomorphisms while not being open or residual. We need the hypothesis of $f$ being $C^{r}$ outside the pseudo-horseshoes in order to guarantee that the entropy is finite. Without this assumption, we don't know how to control the entropy outside the pseudo-horseshoes.
\end{remark}

\






\subsection{Proof of Theorem \ref{thmcomp}}

\noindent

We note that the pseudo-horseshoe construction employed in the proof of the previous theorem provides a robust symbolic dynamics. This enables us to follow the strategy of \cite{RomanaRego2025} and adapt their arguments to the continuous setting to prove the desired result.


\noindent

\begin{proof}[\emph{\textbf{Proof of Theorem \ref{thmcomp}}}]
Let $\mathfrak{R}$ be the set of homeomorphisms that appear in Theorem \ref{thmyano}. To show the persistence of positivity of the entropy, we use the existence of pseudo-horseshoes: Let $f \in \mathfrak{R}$ and let $\Lambda_{k}$ be a pseudo-horseshoe for $f.$ This implies that there exist $g:\Lambda_{k} \to \Sigma$ a homeomorphism, where $\Sigma$ is a full shift of $n_{k}$ symbols and $g$ conjugates $f^{n}$ and the shift $\sigma:\Sigma \to \Sigma,$ for some $n.$ $\sigma$ is expansive and has the shadowing property. In particular, we have that for every $c \in [0,h_{top}(\sigma)]$ we obtain a minimal set $\Lambda_{c} \subset \Sigma$ for $\sigma$ with the property that $h_{top}(\sigma|_{\Lambda_{c}}) = c.$ By minimality, the family $(\Lambda_{c})_{c \in [0,h_{top}(\sigma)]}$ is pairwise disjoint.

Now put $K_{c} = \displaystyle\bigcup_{i = 0}^{n-1}f^{-i}(\Lambda_{i}),$ for every $c \in [0,h_{top}(f)].$ These sets are compact invariants sets of $f$ and they are pairwise disjoint. It is straight-forward to verify that $h_{top}(f|_{K_{c}}) = \dfrac{c}{n}.$ Moreover, since $A_{sing}$ is countable and $[0,h_{top}(\sigma)]$ is uncountable, there exists $c \in (0,h_{top}(\sigma)]$ such that $K_{c} \cap A_{sing} = \varnothing.$ Therefore, $\gamma|_{K_{c}}$ is continuous. This implies that for every $f-$invariant ergodic measure supported on $K_{c}$ such that $h_{\mu}(f) > 0$ (the existence of this measure is guaranteed by the Variational Principle and by the fact that $h_{top}(f|_{K_{c}}) > 0$), $E_{\mu}(\gamma)$ is finite and by Abramov formula,

$$h_{\overline{\mu}}(\psi) = \dfrac{h_{\mu}(f)}{E_{\mu}(\gamma)} > 0.$$

By the Variational Principle, $h_{top}(\psi) > 0.$ And we are done.

\end{proof}

\subsection{Minimal homeomorphisms with Infinite Topological Entropy}

\noindent

As seen in \cite{WTY2009}, when dealing with minimal homeomorphisms, it is very easy to destroy the entropy, we just need to take a sufficiently slow brake near the singularity. To see this, let us recall a result used in \cite{WTY2009}.

\begin{lemma}\label{lemmauniftime}

Let $f:X\to X$ be a minimal homeomorphism on a compact metric space $(X,d).$ For every $\varepsilon > 0$, there exist $L(\varepsilon) \in \N$ such that for every $x,y \in X$, there exists some $0 < l \leq L(\varepsilon)$ such that $f^{l}(y) \in B(x,\varepsilon).$

\end{lemma}

Remember that the function $\gamma$ controls which invariant measures of $f$ can be lifted to invariant measures for the pseudo-singular suspension $\psi.$  Using the theorem above, if $\{p\} = \pi(S)$ then take a non-atomic ergodic measure $\mu$ for $f$ and $x \in X$ such that 

$$\dfrac{1}{n}\sum\limits_{k=0}^{n-1}\gamma(f^{k}(x)) \to \int\gamma d\mu.$$

For $\varepsilon > 0,$ define $M_{\varepsilon} := \inf\limits_{y \in B(p,\varepsilon)} \gamma(y).$  We have

$$\int\gamma d\mu \geq \lim\limits_{n \to \infty} \dfrac{1}{n} \sum\limits_{0 \leq k \leq n-1 : f^{k}(x) \in B(p,\varepsilon)}\gamma(f^{k}(x)) \geq \lim\limits_{n \to \infty} \dfrac{1}{n} \cdot \dfrac{n}{L(\varepsilon)}\cdot M_{\varepsilon} = \dfrac{M_{\varepsilon}}{L(\varepsilon)},$$

and this does not depend on $\mu$ and $x.$ This implies the following proposition.

\begin{proposition}\label{propLe}

If $\liminf\limits_{\varepsilon \to 0}\dfrac{M_{\varepsilon}}{L(\varepsilon)} = \infty$ then all measures of $\psi$ are atomic. In particular, $h_{top}(\psi) = 0.$

\end{proposition}

That is, if we choose an appropriate brake function $\alpha,$ we can readily verify the hypothesis of the proposition above. This is the idea used in \cite{WTY2009}. In particular, the singular suspension flow can be chosen to have zero topological entropy, despite the value of the topological entropy of the minimal homeomorphism. We show the strength of this technique by observing that \textit{there exist minimal homeomorphisms with infinite topological entropy in compact manifolds of any dimension.}


The following theorem, which is our last main result, combines the existence of minimal homeomorphisms with infinite entropy (Theorem \ref{existence of minimal homeo}) with the construction of a pseudo-singular suspension—having a single singularity and zero topological entropy—derived from such a homeomorphism.

\begin{theorem}\label{Existence of minimal homeo 2}
    Given $n \in \N,$ $n \geq 2,$ there exist a compact manifold $M$ and a minimal homeomorphism with $h_{top}(f) = \infty.$ Moreover, there exists a pseudo-singular suspension flow $\psi$ obtained from $f$ with at most one singularity such that $h_{top}(\psi) = 0.$
\end{theorem}

By the work of Grillenberger in \cite{GR72}, as observed in \cite{BKO2019}, there exist topologically weakly mixing minimal Cantor systems with arbitrary entropy in $[0,\infty].$ The next theorem is key to proving the proposition.

\begin{theorem}[\cite{FSF2007}]

Let $R$ be a uniquely ergodic aperiodic homeomorphism of a compact manifold $M$ of dimension $d \geq 2$. Let $h_{C}$ be a homeomorphism on some Cantor space $C$. Then there exists a homeomorphism $f : M \to M$ isomorphic to $R \times h_{C} : M \times C \to M \times C$. Furthermore, the homeomorphism $f$ is a topological extension of $R$: there exists a continuous map $\Phi : M \to M$ such that $\Phi \circ f = R \circ \Phi.$ If $R$ is minimal (resp. transitive), then $f$ can be chosen minimal (resp. transitive).
\end{theorem}
\indent
\begin{proof}[\emph{\textbf{Proof of Theorem \ref{Existence of minimal homeo 2}}}] Let $N \in \N,$ $N \geq 2.$
Choose a uniquely ergodic aperiodic homeomorphism $H$ on the $N-$torus $\mathbb{T}^{N}$ (e.g., an irrational rotation) and a minimal system on Cantor set $h_{C}:C \to C$ with infinite topological entropy (see \cite{GR72}) . By the theorem above, there exists a homeomorphism $f:M \to M$ isomorphic to $H \times h_{C}$ in the measurable sense by a isomorphism $\Phi:M \to M \times C$. 

We have that the topological entropy of $H \times h_{C}$ is infinite. This implies that the topological entropy of $f$ is infinite. This is a consequence of the Variational Principle: Given any invariant measure $\mu$ of $H \times h_{C},$ we have that $\nu := (\Phi^{-1})_{\ast}\mu$ is invariant by $f$ and $(M,\nu,f)$ and $(M\times C,\mu,H \times h_{C})$ are isomorphic in the sense of ergodic theory and then $h_{\nu}(f) = h_{\mu}(H \times h_{C}).$ By the Variational Principle this gives us $h_{top}(f) = \sup\limits_{\nu \in \mathcal{M}^{1}_{f}} h_{\nu}(f) = \sup\limits_{\mu \in \mathcal{M}^{1}_{H\times h_{C}}} h_{\mu}(H\times h_{C}) =\infty.$\\
This proves the existence of minimal homeomorphisms with infinite topological entropy on compact manifolds. Regardless this fact, pseudo-singular suspensions of these homeomorphisms can still have zero entropy: to this end, since $f$ is minimal, the uniform bounded gaps property from Lemma \ref{lemmauniftime} holds; this allows us to choose $\alpha$ as in \cite{WTY2009} such that $\liminf_{\varepsilon \to 0} \dfrac{M_{\varepsilon}}{L(\varepsilon)} = \infty$, satisfying Proposition \ref{propLe}, and consequently the associated pseudo-singular suspension flow $\psi$ has zero topological entropy.
\end{proof}





\bibliography{referencias}
        \bibliographystyle{abbrv}

\noindent \textbf{Jonatas Marinho S. Araujo}\\
National Institute of Pure and Applied Mathematics (IMPA), 22460-320, Brazil\\
jonatas.marinho@impa.br

\ \\

        \noindent \textbf{Sergio Roma\~{n}a}\\
School of Mathematics (Zhuhai), Sun Yat-sen University, 519802, China\\
sergio@mail.sysu.edu.cn

\end{document}